\documentclass[12pt,a4paper]{article}
\usepackage{amsmath}
\usepackage{amsthm}
\usepackage{amssymb}
\usepackage{amsbsy}
\usepackage{amsfonts}
\usepackage{amstext}
\usepackage{amscd}
\usepackage[dvips]{epsfig}

\numberwithin{equation}{section}
\theoremstyle{plain}
\newtheorem{thm}{Theorem}[section]
\newtheorem{prop}[thm]{Proposition}
\newtheorem{cor}[thm]{Corollary}
\newtheorem{lem}[thm]{Lemma}
\theoremstyle{definition}
\newtheorem{exa}[thm]{Example}

\newtheorem{rem}[thm]{Remark}
\newtheorem{defi}[thm]{Definition}

\DeclareMathOperator*{\real}{\mathbb{R}}

\DeclareMathOperator*{\com+}{\mathbb{C}_{+}}
\DeclareMathOperator*{\comp}{\mathbb{C}}
\DeclareMathOperator*{\nat}{\mathbb{N}}
\newcommand{\im}{\text{\normalfont Im~}}
\newcommand{\re}{\text{\normalfont Re~}}
\newcommand{\supp}{\text{supp~}}
\newcommand{\pn}{\mathcal{P}(n)}
\newcommand{\ncpn}{\mathcal{NC}(n)}
\newcommand{\ipn}{\mathcal{I}(n)}
\newcommand{\mpn}{\mathcal{M}(n)}

\newcommand{\comm}{\mathbb{C}_{-}}

\begin{document}
\title{Analytic continuations of Fourier and Stieltjes transforms and generalized moments of probability measures}

\author{Takahiro Hasebe \\ Graduate School of Science,  Kyoto University,\\  Kyoto 606-8502, Japan\\ E-mail: hsb@kurims.kyoto-u.ac.jp}
\date{}

\maketitle

\begin{abstract}
We consider analytic continuations of Fourier transforms and Stieltjes transforms. This enables us to define what we call complex moments for some class of probability measures which do not have moments in the usual sense. There are two ways to generalize moments accordingly to Fourier and Stieltjes transforms; however these two turn out to coincide. As applications, we give short proofs of the convergence of probability measures to Cauchy distributions with respect to tensor, free, Boolean and monotone convolutions. \\


\end{abstract}
\section{Observation on Cauchy distributions} \label{Cauchy}
Let $\mu_{a,b}$ be the Cauchy distribution 
\[
\mu_{a,b}(dx) = \frac{1}{\pi} \cdot \frac{b}{(x-a)^2 + b^2} dx 
\]
with parameters $a \in \real$ and $b > 0$. This distribution has many important properties in probability theory, which we call tensor probability theory,  and also in noncommutative probability theory, especially in free, Boolean and monotone probability theories \cite{Mur3, S-W,V2}. For instance, $\mu_{a,b}$ is an infinitely divisible distribution with respect to tensor, free, Boolean and monotone convolutions. Moreover, $\mu_{a,b}$ is a strictly 1-stable distribution with respect to the four convolutions \cite{Be-Vo, K-G, Has1, S-W}. These facts are proved by calculating the Fourier transform and Stieltjes transform. The calculation for the tensor convolution differs from the other three: it is characterized by the multiplication of Fourier transforms while the other three convolutions are characterized by using Stieltjes transforms (see (\ref{ast})-(\ref{rhd})). In addition,  the Fourier transform strongly differs from the Stieltjes transform: the former is not (real) analytic in general but the latter is always analytic in the complex plane except for the real line. 
In spite of this difference, the four concepts of infinite divisibility have one-to-one correspondence to each other by using the so-called Bercovici-Pata bijections; the reader is referred to \cite{Be-Pa, Has0} and also \cite{Fra1}. A universal role of Cauchy distributions comes from the fact that 
they are fixed by the Bercovici-Pata bijections. 

We note that there are other nontrivial relations involving these four convolutions. For instance, some free $\frac{1}{2}$-stable distributions are infinitely divisible in the tensor sense \cite{S-P} (this is also the case in the monotone case, while its explicit calculation is not in the literature). Gaussian distributions are also such examples: they are infinitely divisible both in the tensor and free senses \cite{B}. These measures are however not fixed points of the Bercovici-Pata bijection.

Now we go on to the idea of this paper. 
The Fourier transform of a probability distribution $\mu$ on the real line is defined by 
\[
\mathcal{F}_\mu (z) = \int_{\real} e^{ixz}\mu(dx),~~~z\in \real, 
\]
and the Stieltjes transform is defined by 
\[ 
G_\mu(z) = \int_{\real} \frac{\mu(dx)}{z - x},~~~ z \in \comp \backslash \real. 
\] 
The reciprocal $F_\mu(z)=\frac{1}{G_\mu(z)}$ is also important in noncommutative probability theory. 
We use the usual notation for convolutions: $\ast$ for the tensor convolution (the usual convolution), $\boxplus$ for the free convolution, $\uplus$ for the Boolean convolution and $\rhd$ for the monotone convolution. These convolutions are characterized by \cite{Be-Vo, Mur3, V2, S-W}
\begin{align}
&\mathcal{F}_{\mu_1 \ast \mu_2}(z) = \mathcal{F}_{\mu_1}(z) \mathcal{F}_{\mu_2}(z), ~~z \in \real, \label{ast}\\
&F_{\mu_1 \boxplus \mu_2} ^{-1} (z) = F_{\mu_1}^{-1}(z) + F_{\mu_2}^{-1}(z) - z, ~~ z \in \Omega_1 \cap \Omega_2, \label{boxplus}\\ 
&F_{\mu_1 \uplus \mu_2}(z) = F_{\mu_1}(z) + F_{\mu_2}(z) - z, ~~ z \in \comm, \label{uplus}\\
&F_{\mu_1 \rhd\mu_2}(z) = F_{\mu_1} (F_{\mu_2}(z)),  ~~z \in \comm, \label{rhd}
\end{align}
for probability measures $\mu_1,~\mu_2$. $F_{\mu_k}^{-1}$ is defined in a domain $\Omega_k$ of $\comm=\{z \in \comp; \im z < 0 \}$ which will be explained in detail later. The upper half plane is usually used in free probability; however we choose $\comm$ to define a correspondence between 
a Fourier transform and a Stieltjes transform. The following observation on Cauchy distributions will be helpful to understand the reason.   

It is well known that $\mathcal{F}_{\mu_{a,b}}$ and $G_{\mu_{a,b}}$ are calculated as 
\begin{align}
&\mathcal{F}_{\mu_{a,b}} (z)= e^{iaz - b|z|},~~~z \in \real, \notag \\
&G_{\mu_{a,b}}(z) = \frac{1}{z - a +ib ~\text{sign} (\im z)},  ~~~ z \in \comp \backslash \real,  \notag
\end{align}
where $\text{sign}(x) = 1$ for $x > 0$ and $\text{sign}(x) = -1$ for $x < 0$. 
$\{ \mu_{at, bt} \}_{t \geq 0}$ is easily shown to become a convolution semigroup in tensor, free, Boolean and monotone senses, 
and therefore $\mu_{a,b}$ is infinitely divisible.  

Moments of a probability measure can be obtained as the coefficients of the Taylor expansion of $\mathcal{F}_\mu (z)$ or the Laurent expansion of 
$G_\mu (z)$. More precisely, let $m_n(\mu)$ be the $n$th moment of $\mu$ and then  
\begin{align}
&\mathcal{F}_{\mu} (z)= \sum_{n=0}^\infty \frac{m_n(\mu)}{n!}(iz)^n, \notag \\
&G_{\mu}(z) = \sum_{n=0}^\infty \frac{m_n(\mu)}{z^{n+1}}.   \notag
\end{align}
We neglected here the convergence of the series.

Let $\com+$ be the upper half plane. $G_{\mu_{a,b}}$ is analytic in $\comm$ and it has analytic continuation from $\comm$ to $\comp \backslash \{a + ib \}$. We denote this function by $\widetilde{G}_{\mu_{a,b}}$. It is important that $G_{\mu_{a,b}} \neq \widetilde{G}_{\mu_{a,b}}$ in $\com+$.  
We then have the convergent series 
\[
\begin{split}
\widetilde{G}_{\mu_{a,b}}(z) &= \frac{1}{z(1 - \frac{a+ib}{z})} \\
          &= \sum_{n=0} \frac{(a+ib)^n}{z^{n+1}} 
          \end{split}
\]
for $|z| > \sqrt{a^2 + b^2}$. At first sight a similar idea seems impossible for the Fourier transform because of the absolute value of $z$. 
However, if we restrict the domain to $z > 0$, then  $\mathcal{F}_{\mu_{a,b}} (z)= e^{iaz - bz} = e^{i(a+ib)z}, ~~z > 0$. We denote its analytic continuation to $\comp$ by $\widetilde{F}_{\mu_{a,b}}$: 
\[
\widetilde{F}_{\mu_{a,b}} = \sum_{n=0} \frac{(a+ib)^n}{n!}(iz)^n. 
\]
Since the same coefficient $(a+ib)^n$ appears in both Fourier and Stieltjes transforms, we decide to call $m_n(\mu_{a,b}):=(a+bi)^n$ the
 \textit{complex moments} of the Cauchy distribution. In the next section, we extend the concept of complex moments to more general probability measures. 

This paper is organized as follows.  The most important class of probability 
measures to be considered is $\mathcal{P}_1$ which roughly consists of probability measures with analytic densities at infinity.  
In Section \ref{Genmom}, we clarify similarity  between the Fourier transform and the Stieltjes transform regarding analytic properties, extending the definition of moments and four kinds of cumulants to $\mathcal{P}_1$. More precisely, we generalize moments in two different ways: one by means of analytic continuation of the Fourier transforms from $z > 0$ to $\comp$, and the other by analytic continuation of Stieltjes transforms from $\comm$ to $\comm \cup \{z \in \comp; |z|>R \}$ for some $R>0$. While these transforms have different analytic properties, such defined moments turn out to coincide. Then four kinds of cumulants can 
be defined and satisfy the moment-cumulant formulae. These moments and cumulants are complex numbers and 
we call them \textit{complex moments} and \textit{complex cumulants} respectively. 
Then the Cauchy distributions turn out to have the same complex cumulants in any sense of tensor, free, Boolean and monotone. This reveals a universal role of Cauchy distributions. In Section \ref{Lim} we prove that the sum of i.i.d. random  variables divided by the number of them converges to a Cauchy distribution, using complex moments and cumulants. This holds in any case of tensor, free, Boolean and monotone convolutions. In particular the monotone case is quite new since a complex analytic method is not known to examine limit behavior of monotone convolutions of probability measures with infinite moments. In Section \ref{Cha} we characterize $\mathcal{P}_1$ in terms of Fourier transforms and Stieltjes transforms; the characterization by Fourier transforms results in a kind of Paley-Wiener theorem.  

\section{Complex moments} \label{Genmom}
We introduce a class $\mathcal{P}_1$ of probability measures $\mu$ of the forms 
\begin{equation}\label{decomp}
\mu= \nu + \lambda, 
\end{equation} 
where $\nu$ and $\lambda$ satisfy the following properties. 
\begin{itemize}
\item[(A1)] $\nu$ is a positive measure with compact support.  
\item[(A2)] There exist $0 < r < R$ and a real analytic function $p(x) = \sum_{n=2}^\infty \frac{a_n}{x^n}$ which is absolutely convergent for 
 $|x|>r$ and is non-negative on $|x|\geq R$ such that $\lambda(dx) = p(x)\textbf{1}_{|x|\geq R}(x)dx$.      
\end{itemize}
We also define $\mathcal{P}_2$ consisting of $\mu = \nu + \lambda$, where 
$\nu$ is a positive measure with finite second moment and $\lambda$ satisfies the property (A2). 

\begin{exa}
The following measures satisfy the property (A2). 
\begin{itemize}
\item[(a)] Cauchy distributions $\lambda = \mu_{a,b}$ defined on $|x| \geq R$ with $R > |a| + b$. 
\item[(b)] $\lambda = \frac{1}{\sqrt{1 + x^4}}dx$  on $|x| \geq R$
 with $R > 1$.  
%\[
%\mu(dx) = \frac{1}{(1 + P(x))^\alpha}dx, 
%\]
%where $P(x)$ is a polynomial satisfying $P(0) = 0$, $P(x) \geq 0$ for $|x| \geq R$ and $\text{deg~}P $
% Let $f(x)$ be an real analytic function for $|x| > r$ of the form 
%\[
%f(x) = \sum_{n =2}^\infty \frac{a_n}{x^n},~~~|x| > r. 
%\]
%Let $g$ be an analytic function at $a$ and  
\item[(c)] $\lambda =  \frac{P(x)}{Q(x)}dx$ on $(-\infty,-R] \cup [R, \infty)$ for sufficiently large $R>0$, where  $P(x)$ and $Q(x)$ are polynomials satisfying that $\text{deg~}P(x)+2 \leq \text{deg~}Q(x)$ and $P(x)$, $Q(x) \geq 0$ for $|x| > R$.   
\item[(d)] $\lambda = \frac{1}{(1+x^{2n})^{\frac{p}{n}}}dx$ on $|x| \geq R$ with $R>1$, where $n \geq 1$ and $p \geq 1$ are natural numbers. \end{itemize}
(c) and (d) are generalizations of (a) and (b) respectively. In (a), we can define $\nu$ to be the restriction of the Cauchy distribution $\mu_{a,b}$ on 
$[-R, R]$, and then $\mu_{a,b} = \nu + \lambda$ belongs to $\mathcal{P}_1$.  

We check the property (A2) for the Cauchy distribution $\mu_{0, 1}$ for instance. The density function is given by   
\[\begin{split}
\frac{1}{\pi(1+x^2)} &= \frac{1}{\pi x^2(1 + \frac{1}{x^2})} \\
                     &= \frac{1}{\pi}\sum_{n = 0} ^\infty \frac{(-1)^n}{x^{2n+2}}.   
\end{split}
\]
This series converges for $|x| > 1$. 
\end{exa}

The free, Boolean and monotone convolutions are all characterized by using Stieltjes transforms. Therefore, we consider the analytic continuation of Stieltjes transforms. 

\begin{prop} \label{prop1}
Let $\lambda$ be a positive measure on $(-\infty,-R] \cup [R, \infty)$ which satisfies (A2). Then 
$G_\lambda (z)$ defined in $\comm$ has analytic continuation to $\comp_{-} \cup \{z \in \comp; |z| > R \}$. We denote it by $\widetilde{G}_\lambda$.  $\widetilde{G}_\lambda(z)$ has convergent series of the form 
\[
\widetilde{G}_\lambda(z) = \sum_{n=0}^\infty \frac{b_n}{z^{n+1}}, ~~~|z| > R.    
\]
\end{prop}
\begin{proof}
By assumption $\frac{1}{w^2}p(\frac{1}{w})$ is analytic in $\{z; |z| < r^{-1} \}$. We define $v = \frac{1}{z}$ for $\im z < 0$. We notice that $\im v > 0$. 
Then it holds that 
\begin{equation}\notag
\begin{split}
G_\lambda(z) &= \int_{|x|>R} \frac{p(x)}{z-x}dx \\ 
           & = \int_{-R^{-1}}^{R^{-1}} \frac{p(\frac{1}{u})}{z-\frac{1}{u}}\frac{1}{u^2}du \\ 
           &= v\int_{-R^{-1}}^{R^{-1}} \frac{u p(\frac{1}{u})}{u-v}\frac{1}{u^2}du \\  
&= v\int_{C_R} \frac{w p(\frac{1}{w})}{w-v}\frac{1}{w^2}dw,    
\end{split} 
\end{equation}
where $C_R$ is the semicircle $C_R =\{R^{-1}e^{i \theta}; - \pi \leq \theta \leq 0 \}$ with the counterclockwise direction. The last expression is also valid for $|v| < R^{-1}$ as well as for $v \in \com+$. The Taylor expansion of 
it becomes 
\[
\sum_{n=0}^\infty v^{n+1}\int_{C_R}\frac{p(\frac{1}{w})}{w^{n+2}}dw. 
\]
\end{proof}

The tensor convolution (the usual convolution) is characterized by the multiplication of Fourier transforms.  
Therefore, we in turn consider analytic continuation of Fourier transforms.

\begin{prop}\label{prop2}
Let $\lambda$ be a positive measure on $(-\infty,-R] \cup [R, \infty)$ which satisfies (A2). Then 
$\mathcal{F}_\lambda$ defined for $z > 0$ is the restriction of an entire function, which we denote by $\widetilde{\mathcal{F}}_\lambda$. 
Therefore the Taylor series 
\[
\widetilde{\mathcal{F}}_\lambda(z) = \sum_{n=0}^\infty \frac{c_n}{n!}(iz)^n    
\]  
is convergent for all $z \in \comp$. 
\end{prop}
\begin{proof}
We have for $z > 0$
\begin{equation}\label{eq12}
\begin{split}
\mathcal{F}_\lambda (z) &= \int_{|x|>R} e^{ixz}p(x)dx \\
   &= \int_{-R^{-1}}^{R^{-1}}e^{i\frac{z}{u}}
p\Big(\frac{1}{u}\Big)\frac{1}{u^2}du.  
\end{split} 
\end{equation}
Since $\re \!\big( i\frac{z}{u}\big) \leq 0$ for $\im u \leq 0$, $u \neq 0$, we have 
\[
\lim_{\varepsilon \searrow 0} \int_{C_{\varepsilon^{-1}}}e^{i\frac{z}{u}}
p\Big(\frac{1}{u}\Big)\frac{1}{u^2}du = 0. 
\]
Then the path in the integral (\ref{eq12}) can be changed as 
\[
\begin{split}
\int_{-R^{-1}}^{R^{-1}}e^{i\frac{z}{u}}
p\Big(\frac{1}{u}\Big)\frac{1}{u^2}du &= \lim_{\varepsilon \searrow 0} \Big( \int_{-R^{-1}}^{-\varepsilon} + \int^{R^{-1}}_{\varepsilon}\Big) \Big(e^{i\frac{z}{u}}
p\Big(\frac{1}{u}\Big)\frac{1}{u^2}du\Big)\\ 
&= \lim_{\varepsilon \searrow 0} \int_{\Gamma_{R, \varepsilon}}e^{i\frac{z}{u}}
p\Big(\frac{1}{u}\Big)\frac{1}{u^2}du \\
  &= \lim_{\varepsilon \searrow 0} \int_{C_{R}} e^{i\frac{z}{u}}
p\Big(\frac{1}{u}\Big)\frac{1}{u^2}du\\ 
&= \int_{C_{R}}e^{i\frac{z}{u}}
p\Big(\frac{1}{u}\Big)\frac{1}{u^2}du, 
\end{split}
\]
where $\Gamma_{R, \varepsilon}$ is a curve consisting of two line segments 
$[-R^{-1}, -\varepsilon]$ and $[\varepsilon, R^{-1}]$ and a semicircle $C_{\varepsilon^{-1}}$.  
Therefore we have 
\begin{equation}
\mathcal{F}_\lambda (z) = \int_{C_{R}}e^{i\frac{z}{u}}
p\Big(\frac{1}{u}\Big)\frac{1}{u^2}du, 
\end{equation}
for $z > 0$. This expression continues $\mathcal{F}_\lambda$ from $(0, \infty)$ to $\comp$. 
\end{proof}

It is natural to define what we call complex moments of $\lambda$ by the sequence $b_n$ for free, Boolean and monotone probability theories and by $c_n$ for tensor probability theory. There is no a priori evidence that these two ways define the same quantities.  
However, the two analytic continuations give the same complex moments as is understood from the proofs: 

\begin{thm}\label{thm1}
Let $b_n$ and $c_n$ be complex numbers defined in the propositions \ref{prop1} and \ref{prop2}. Then $b_n = c_n$ for all $n \geq 0$ and they are expressed as  
\[
b_n = c_n = \int_{\gamma_R} z^n p(z)dz.  
\]
where $\gamma_R =\{Re^{i \theta}; 0 \leq \theta \leq \pi \}$ equipped with counterclockwise direction. 
We denote $b_n$ by $m_n(\lambda)$.    
\end{thm}
\begin{proof}
We have proved that $b_n = c_n =\int_{C_R}\frac{p(\frac{1}{w})}{w^{n+2}}dw$. We can easily prove that this integral is equal 
to $\int_{\gamma_R} z^n p(z)dz$.   
\end{proof}
Therefore we can generalize moments both in terms of the Fourier transform and the Stieltjes transform. 
% suitable for tensor, free, Boolean and monotone convolutions at the same time.  

\begin{defi} 
(Complex moments). We define the $n$th \textit{complex moment} of $\mu \in \mathcal{P}_1$ by $m_n(\mu) = m_n(\nu) + m_n(\lambda)$, where $\nu$ and $\lambda$ are given by a decomposition in (\ref{decomp}). Similarly we define the first and second moments of $\mu \in \mathcal{P}_2$ by $m_j(\mu) = m_j(\nu) + m_j(\lambda)$ for $j = 1$, $2$. 
\end{defi}
\begin{rem}
(1) The decomposition in (\ref{decomp}) and the choice of the convergence radius 
$R$ of $\lambda$ are not unique. It is however not difficult to prove that the definition of complex moments does not depend on these ambiguities. \\
(2) We notice that the domain $(0, \infty)$ of the Fourier transform and the domain $\comm$ of the Stieltjes transform play analogous roles 
in the analytic continuations. The information of $G_\mu$ on $\com+$  is recovered from the relation $G_\mu(z) = \overline{G_\mu(\bar{z})}$; 
this is similar to the relation $\mathcal{F}_\mu(z) = \overline{\mathcal{F}_\mu(-z)}$. 
\end{rem}
\begin{exa}
(1) $\mu(dx) = \frac{2}{\pi(1+4x^4)}dx$ on $\real$. Then $\mathcal{F}_\mu(z) = \sqrt{2}\re \! (e^{\frac{i-1}{2}z}e^{-\frac{\pi i}{4}}) = e^{-\frac{z}{2}}(\cos(\frac{z}{2}) + \sin(\frac{z}{2}))$ for $z > 0$. The $n$th complex moment is $m_n(\mu)=i^n 2^{\frac{1-n}{2}}\sin(\frac{1-n}{4}\pi)$. In particular, $m_1(\mu)=0$, $m_2(\mu) = \frac{1}{2}$, $m_3(\mu)=\frac{i}{2}$. 
\\
(2) $\mu_0(dx) = \frac{\sqrt{2}x^2}{\pi(1+x^4)}dx$ on $\real$. Then $\mathcal{F}_{\mu_0}(z) = \sqrt{2}\re \! (e^{\frac{\pi i}{4}}\exp(ize^{\frac{\pi i }{4}})) = \sqrt{2} \sum_{n=0}^\infty \frac{\cos(\frac{3n +1}{4}\pi)}{n!}z^n$ for $z > 0$.  The $n$th complex moment is given by $m_n(\mu_0) = \sqrt{2}i^n \cos(\frac{n-1}{4}\pi)$. In particular, $m_1(\mu_0) = \sqrt{2}i$, $m_2(\mu_0) = -1$. \\ 
(3) More generally, $\mu_a = \frac{\sqrt{2}(x-a)^2}{\pi(1+(x-a)^4)}dx$ on $\real$, $a \in \real$. This is the convolution of $\mu_0$ and $\delta_a$. Since the binomial expansion holds for the moments of the tensor convolution,  the $n$th complex moment is $m_n(\mu_a) = \sum_{k=0}^n \frac{n!}{k!(n-k)!}\sqrt{2}i^k \cos(\frac{k-1}{4}\pi) a^{n-k}$. 
In particular, $m_1(\mu_a) = a + \sqrt{2}i$, $m_2(\mu_a) = -1+a^2 + 2\sqrt{2}ai$.  

In general, if $\mu \in \mathcal{P}_1$ is symmetric, there exist $r_n \in \real$ such that $m_n(\mu) = r_n i^n$. 
This is because the Fourier transform takes real values on $\real$ and therefore the coefficient of $z^n$ ($z > 0$) in the Taylor expansion is  a real number for any $n$.  
\end{exa}

In addition to moments, we can extend cumulants. 
We need the following fact whose proof is identical to that in \cite{H-S}.

\begin{lem} For any $\mu \in \mathcal{P}_1$, 
$m_n(\mu^{\ast N})$, $m_n(\mu^{\boxplus N})$, $m_n(\mu^{\uplus N})$ and $m_n(\mu^{\rhd N})$ are all polynomials of $N$ and 
$m_k(\mu)$ ($1 \leq k \leq n$). 
\end{lem}

\begin{defi}
(Complex cumulants). We define the $n$th \textit{complex tensor cumulant} $K_n^T(\mu)$, \textit{complex free cumulant} $K_n^F(\mu)$, \textit{complex Boolean cumulant} $K_n^B(\mu)$ and \textit{complex monotone cumulant} $K_n^M(\mu)$ of $\mu \in \mathcal{P}_1$ respectively by the coefficients of $N$ in $m_n(\mu^{\ast N})$, $m_n(\mu^{\boxplus N})$, $m_n(\mu^{\uplus N})$ and $m_n(\mu^{\rhd N})$. If there is no confusion, we call them cumulants for simplicity. 
\end{defi}
%\begin{rem} Let $X$ be a random variable with distribution $\mu \in \mathcal{P}_1$. Then the probability measure $\mu^{\ast N}$ corresponds to the distribution of the dot operation $N.X$ for tensor independence, for instance \cite{H-S}. We however use only probability measures because of the difficulty of unbounded operators. 
%\end{rem}
We find that the four kinds of cumulants of the Cauchy distribution $\mu_{a,b}$ are given by 
\[
K_n^T(\mu_{a,b}) = K_n^F(\mu_{a,b}) = K_n^B(\mu_{a,b})=K_n^M(\mu_{a,b})=\begin{cases} &a+bi,~~n=1, \\ & 0,~~n \geq 2. \end{cases} 
\]
This helps us to understand a universal role of Cauchy distributions in noncommutative probability theory.

From now on we show basic properties of complex moments and cumulants. 

\begin{prop}\label{prop11}
$\im \!(m_1(\mu)) \geq 0$ for all $\mu \in \mathcal{P}_2$. 
\end{prop}
\begin{proof}
Let $\mu = \nu + \lambda$ be a decomposition in (\ref{decomp}). 
We have only to consider $\lambda$. By Theorem \ref{thm1} $m_1(\lambda)$ is calculated as 
\[
\begin{split}
m_1(\lambda) &= \int_{\gamma_R} z p(z)dz \\
              &= \sum_{n = 2}^\infty a_n \int_{\gamma_R} z^{1-n}dz \\
              &= i \sum_{n = 2}^\infty a_n \int_{0}^\pi R^{2-n}e^{i(2-n)\theta} d\theta \\ 
  &= i\pi a_2 - \sum_{n = 3}^\infty a_n  R^{2-n}\frac{1 - (-1)^n}{2-n}.
\end{split}
\]
Since $p(x)$ is a positive function, $a_2 \geq 0$. Therefore, 
$\im \!(m_1(\mu)) = \pi a_2 \geq 0$.  
\end{proof}
The above fact will be used in Section \ref{Lim} to formulate convergence of probability measures to Cauchy distributions.  

The following property follows from the integral representation of complex moments in Theorem \ref{thm1}. This result means that the set $\mathcal{P}_1$ is analogous to the set of probability measures with compact supports. 
\begin{prop}\label{prop13}
Let $\mu = \nu + \lambda \in \mathcal{P}_1$ be a decomposition as in (\ref{decomp}), where 
$\nu$ is supported on $[-R_1, R_1]$ and $\lambda$ is supported on $(-\infty,-R_2] \cup [R_2, \infty)$.   
Then there exists $C>0$ depending on $\nu,\lambda$ such that $|m_n(\mu)| \leq C R^n$ for all $n \geq 1$, 
where $R:=\max \{R_1, R_2\}$. 
\end{prop}
\begin{proof} Let $p$ be the density function of $\lambda$.  We have 
\[
|m_n(\nu)| \leq \int_{-R_1} ^{R_1} |x|^n \nu(dx) \leq \nu(\real) R_1 ^n
\] 
and 
\[
|m_n(\lambda)| \leq \int_{\gamma_{R_2}} |z|^n|p(z)||dz| \leq CR_2 ^{n+1} 
\]
for a constant $C>0$. This completes the proof.  
\end{proof}
\begin{rem}
In Section \ref{Cha} we introduce a canonical constant $R_\mu$ of $\mu \in \mathcal{P}_1$. Then we can take the constant $R$ arbitrarily near to $R_\mu$: $R_\mu = \inf \{R \geq 0; \sup_{n \geq 1} \frac{|m_n(\mu)|}{R^n} < \infty \}$. 
\end{rem}

Let $\pn$, $\ncpn$, $\ipn$ and $\mpn$ be respectively the sets of partitions, non-crossing partitions, interval partitions and monotone partitions of $\{1, \cdots, n \}$. Only the monotone partitions are equipped with order structure.  The reader is referred to \cite{H-S, N-S1} for details. 
For a partition $\pi =\{V_1, \cdots, V_{k}\}\in \pn$ and a sequence $\{ r_n \}_{n \geq 1} \subset \comp$, we define $r_\pi:= r_{|V_1|} \cdots r_{|V_k|}$.  

Since we defined complex moments via analytic continuations, the following properties hold. 
\begin{prop}
The moment-cumulant formulae hold:   
\begin{align}
&m_n (\mu) = \sum_{\pi \in \pn} K_\pi^T (\mu), \label{cmc} \\
&m_n(\mu) = \sum_{\pi \in \ncpn} K_\pi^F(\mu), \label{fmc} \\
&m_n (\mu) = \sum_{\pi \in \ipn} K_\pi^B(\mu), \label{bmc} \\
&m_n (\mu)= \sum_{(\pi, \lambda) \in \mpn} \frac{K_\pi^M(\mu)}{|\pi|!}. \label{mmc}   
\end{align}
\end{prop}
The proof is the same as the usual case where $\mu$ has a compact support; the reader is referred to \cite{H-S, Shi, Spe2,S-W} and also 
\cite{H-S2} for a unified treatment of (\ref{cmc}), (\ref{fmc}) and (\ref{bmc}).

Let $D_c$ be the dilation operator defined by 
\[
\int_{\real} f(x) (D_c \mu)(dx) = \int_{\real}f(c x) \mu(dx)
\]
for all bounded continuous function $f$.

\begin{prop}\label{prop3} Let $\mu, ~\mu_1, ~\mu_2 \in \mathcal{P}_1$. 
The complex cumulants are additive and homogeneous: \\
(1) $K_n^T(\mu_1 \ast \mu_2) = K_n^T(\mu_1) + K_n^T(\mu_2)$,  ~~$K_n^T (D_c \mu) = c^n K_n^T(\mu)$, \\
(2) $K_n^F(\mu_1 \boxplus \mu_2) = K_n^F(\mu_1) + K_n^F(\mu_2)$,  ~~$K_n^F (D_c \mu) = c^n K_n^F(\mu)$, \\ 
(3) $K_n^B(\mu_1 \uplus \mu_2) = K_n^B(\mu_1) + K_n^B(\mu_2)$,  ~~$K_n^B (D_c \mu) = c^n K_n^B(\mu)$, \\
(4) $K_n^M(\mu ^{\rhd N}) = NK_n^M(\mu)$ for any $N \in \nat$,  ~~$K_n^M (D_c \mu) = c^n K_n^M(\mu)$.     
\end{prop}
\begin{rem}
As proved in Corollary \ref{cor1}, $\mathcal{P}_1$ is closed under the four kinds of convolutions, and hence the above cumulants such as $K_n^T(\mu_1 \ast \mu_2)$ are well-defined.  
\end{rem}
This proof is also the same as the usual one. See \cite{H-S, H-S2, Shi, Spe2,S-W} for details.

\section{Convergence of probability measures to Cauchy distribution}\label{Lim}
Now we prove that the distribution of $\frac{X_1 + \cdots + X_N}{N}$ converges weakly to Cauchy distributions for i.i.d. random variables $X_i$ with the distribution in $\mathcal{P}_1$ or $\mathcal{P}_2$.  We consider only probability measures since difficulty arises to formulate the results in terms of unbounded operators.  
It is known that $\mathcal{P}_2$ belong to the domain of attractor of Cauchy distributions in tensor, free and Boolean probability theories \cite{Be-Pa, K-G}. 
%. In the Boolean case, an explicit statement cannot be found in the literature, but the technique in \cite{Wa} will be applicable to show the limit theorem. 
There are however two merits in the approach of this paper: the proof is simple; the two parameters $a$ and $b$ appearing in the limit distributions can be explicitly calculated in terms of complex moments.  

Before proving them, we introduce a tool in free probability.  
A domain of $F^{-1}_\mu$ for a probability measure $\mu$ is defined as follows.  
It is known that there exists a domain $\Omega$ such that $F_\mu$ has the right inverse $F_\mu^{-1}$ on $\Omega$. It is possible to choose $\Omega$ of the form $\Omega = \bigcup_{\alpha > 0} \Gamma_{\alpha, \beta_\alpha}$ for some $\beta_{\alpha} > 0$, where $\Gamma_{\alpha, \beta}=\{z \in \comm;  |\re z| < \alpha |\im z|,~\im z < - \beta \}$.  

\begin{thm}
Let $\mu$ be a probability measure in $\mathcal{P}_2$,  $a = \re \!\big( m_1(\mu)\big)$ and $b = \im \!\big( m_1 (\mu)\big)$.  
%and let $X_j$ $(j = 1, 2, \cdots)$ be independent, identically distributed random variables with distribution $\mu$. Then the distribution of $\frac{X_1 + \cdots + X_N}{N}$ 
The probability measures $\mu_{N}^T \equiv (D_{\frac{1}{N}}\mu)^{\ast N}$, $\mu_{N}^F \equiv (D_{\frac{1}{N}}\mu)^{\boxplus N}$ 
and $\mu_{N}^B \equiv (D_{\frac{1}{N}}\mu)^{\uplus N}$ all converge weakly to the Cauchy distribution $\mu_{a,b}$.  
If $b = 0$, we understand that $\mu_{a,0}= \delta_a$. 
\end{thm}
\begin{proof}
As usual we take a decomposition $\mu = \nu + \lambda$ in (\ref{decomp}). 
%Let $X^{(N)}$ be the random variable defined by $\frac{X_1 + \cdots + X_N}{N}$ and  $\mu = \nu + \lambda$ be a decomposition in (\ref{decomp}). 
It is well known that $\int_{\real} e^{ixz}\nu(dx) = \nu(\real) + im_1(\nu)z + o(z)$ as $z \to 0$. Therefore we have  
\begin{equation}\notag
\begin{split}
\mathcal{F}_{\mu_N^T}(z) &= \Big(\mathcal{F}_{\mu}\Big(\frac{z}{N}\Big)\Big) ^{N} \\
              &= \Big(1 + \frac{m_1(\mu)}{N}(iz) + o(N^{-1}) \Big) ^N \\
             & \to e^{im_1(\mu)z}~~\text{as~}N \to \infty \\
              & = e^{i az - bz}
              \end{split} 
\end{equation}
for $z \geq 0$. We notice that $m_1(\mu)$ has a non-negative imaginary part from Proposition \ref{prop11}.  
Therefore the limit is the Fourier transform of the Cauchy distribution 
$\mu_{a,b}$ with $a = \re \!(m_1(\mu))$ and $b= \im \!(m_1(\mu))$. Since $\mathcal{F}_{\rho}(z) = \overline{\mathcal{F}_\rho(-z)}$ for $z < 0$, $\mathcal{F}_{\mu_N^T}(z) $ converges to $\mathcal{F}_{\mu_{a,b}}(z)$ for all $z \in \real$. 
Therefore $\mu_N^T$ converges to $\mu_{a,b}$ weakly. 

We use the characterization (\ref{boxplus}) to prove the weak convergence of $\mu_N^F$.  
It is easy to prove that 
%\[
%F_\nu (z) = z - m_1(\nu) + f_\nu(z), ~~f_\nu(z) \to 0 ~~\text{as~} |z| \to \infty \text{non tangentially satisfying} z \in \comp_{-}.  
%\]
\[
F_\mu(z)= z -m_1(\mu) +f(z),~~ f(z) \to 0 \text{~as~} z \to \infty  \text{~non tangentially satisfying~} z \in \comm,   
\] 
by using the Nevanlinna representation $F_\nu(z) = z - m_1(\nu) + \int_{\real}\frac{1}{x-z}\tau(dx)$ for $\nu$, where $\tau$ is a positive finite measure \cite{Akh}. 
The right inverse can be written as $F_\mu^{-1}(z) = z +m_1(\mu) + g(z)$. We show that $g(z) \to 0$ in the non tangential limit $z \to \infty$. 
The equality $F_\mu (F_\mu^{-1}(z)) = z$ for $z \in \Omega$ becomes 
\[
0 = g(z) + f(z + m_1(\mu) + g(z)). 
\]
Since $\frac{F_\mu^{-1}(z)}{z} \to 1$ in the non tangential limit $z \to \infty$ (Corollary 5.5 in \cite{Be-Vo}), 
we get $\frac{g(z)}{z} \to 0$. Then $ f(z + m_1(\mu) + g(z)) = f(z (1 + \frac{m_1(\mu)}{z} + \frac{g(z)}{z}) ) \to 0$. 
Therefore, $g(z) \to 0$ in the non tangential limit.  
Finally we obtain 
\[
\begin{split}
F^{-1}_{\mu_N^F}(z) &= F^{-1}(Nz) - (N-1)z \\
                    &=  z + m_1 (\mu) + g(Nz) \\
                &\to z + m_1(\mu)   
                    \end{split}
                    \] 
as $N \to \infty$ for all $z \in \comm$.     
Thus $\mu_N^F$ converges to $\mu_{a, b}$ from Proposition 5.7 in \cite{Be-Vo}. 

The Boolean case is easier. 
(\ref{uplus}) implies that 
\[
F_{\mu_N ^B} (z) = F_{\mu}(Nz) - (N-1)z \to F_{\mu_{a,b}}(z) = z - m_1(\mu) ~~\text{as}~ N \to \infty. 
\]
The weak convergence follows from Theorem 2.5 in \cite{Maa}. 
\end{proof}

We can prove a similar result for the monotone convolution. This is quite new since there are not complex analytic methods to analyze iteration of monotone convolutions of probability measures without the usual moments. 
Now we prove a limit theorem using the complex moments. In this case we restrict the class of probability measures to $\mathcal{P}_1$. 
We need a Lemma. 
\begin{lem}\label{lem1}
If $|m_n(\mu)| \leq R^n$ for all $n \geq 0$, then $|m_n(\mu^{\rhd N})| \leq (NR)^n$ for all $n,~N \geq 0$. 
\end{lem}
\begin{proof}
We prove this by induction. Assume that $|m_n(\mu^{\rhd N})| \leq (NR)^n$ for all $n \geq 0$. 
Since we complex moments via analytic continuation, it holds that 
\begin{equation}\label{moment}
\begin{split}
m_n(\rho \rhd \sigma) &= \sum_{k = 0} ^{n} \sum_{\substack{j_0 + j_1 + \cdots +j_k = n - k, \\  0 \leq j_l, ~0 \leq l \leq k }} m_k(\rho) m_{j_0}(\sigma)\cdots m_{j_k}(\sigma)   
\end{split}
\end{equation} 
for $\rho,~\sigma \in \mathcal{P}_1$ \cite{H-S}.  If $|m_n(\rho)| \leq R_1^n$ and $|m_n(\sigma)| \leq R_2^n$ for all $n$, then we have 
\begin{equation}\notag
\begin{split}
|m_n(\rho \rhd \sigma)| &\leq \sum_{k = 0} ^{n} \sum_{\substack{j_0 + j_1 + \cdots +j_k = n - k, \\  0 \leq j_l, ~0 \leq l \leq k }} R_1 ^k R_2 ^{n-k} \\
           &= (R_1 + R_2)^n,  
\end{split}
\end{equation}
since $\sum_{\substack{j_0 + j_1 + \cdots +j_k = n - k, \\  0 \leq j_l, ~0 \leq l \leq k }} = \frac{n!}{k!(n-k)!}$.  
We replace $\rho$ by $\mu$ and $\sigma$ by  $\mu^{\rhd N}$, and then the above inequality becomes 
\[
|m_n(\mu^{\rhd N+1})| \leq (N+1)^n R^n. 
\]
\end{proof}

\begin{thm}
Let $\mu$ be a probability measure in $\mathcal{P}_1$ with $a = \re \!\big( m_1(\mu)\big)$ and $b = \im \!\big( m_1 (\mu)\big)$.  
Then $\mu_N^M := (D_\frac{1}{N}\mu)^{\rhd N}$ converges weakly to the Cauchy distribution $\mu_{a,b}$. If $b=0$ we understand that 
$\mu_{a,0} = \delta_a$. 
\end{thm}
\begin{proof}
Applying \ref{prop3}, we get 
$K_1^M(\mu_N^M) = a + bi$ and $K_n^M(\mu_N^M)=N^{-(n-1)}K_n^M(\mu) \to 0$ as $N \to \infty$ for $n \geq 2$. The limit complex cumulants are the complex monotone cumulants of 
the Cauchy distribution $\mu_{a,b}$, and therefore $m_n(\mu_N^M)$ converges to $(a+bi)^n$ as $N \to \infty$. 
Therefore the generating function $G_{\mu_N^M}(z) = \sum_{n=0} ^\infty\frac{m_n(\mu_N^M)}{z^{n+1}}$ converges to 
$G_{\mu_{a,b}}(z) =\sum_{n=0} ^\infty\frac{(a+bi)^n}{z^{n+1}}$ in the sense of formal power series. 
Moreover, the estimation 
\[
\Big|\frac{m_n(\mu_N^M)}{z^{n+1}}\Big| \leq \frac{R^n}{|z|^{n+1}}
\]
holds for some $R > 0$ from Lemma \ref{lem1} and Proposition \ref{prop3}; this implies that $G_{\mu_N^M}(z)$ converges to $G_{\mu_{a,b}}(z)$ uniformly in $|z|\geq 2R$. 
Thus we obtain the weak convergence $\mu_N^M \to \mu_{a,b}$ using Theorem 2.5 in \cite{Maa}. 
\end{proof}

\section{Characterization of the set $\mathcal{P}_1$}\label{Cha}
In this section we characterize the class $\mathcal{P}_1$ in terms of analytic properties of $G_\mu(z)$ and its reciprocal $F_\mu(z) = \frac{1}{G_\mu(z)}$, and moreover in terms of $F_\mu^{-1}$ and Fourier transforms $\mathcal{F}_\mu(z)$.  Similarity between $\mathcal{P}_1$ and probability measures with compact supports will be understood from the characterizations; in particular, characterization by means of Fourier transforms results in a kind of Paley-Wiener theorem.  

We recall that $\mu \in \mathcal{P}_1$ has a density of the form $\sum_{n=2}^\infty \frac{a_n}{x^n}$ for sufficiently large $|x|$. 
Let $R_\mu$ be defined by 
\[R_\mu = \inf \{R >0; \mu|_{|x|\geq R}(dx) = \sum_{n=2}^\infty \frac{a_n}{x^n} dx \}. 
\] 
For instance, $R_{\mu_{0,b}} = b$. 
$\mu$ is compactly supported if and only if $a_n = 0$ for all $n$. If $\mu$ is compactly supported, then $R_\mu$ is 
equal to $\inf \{R>0, \supp \mu \subset [-R, R] \}$. We will prove that $R_\mu$ plays a role similar to the compactly supported case. 

We introduce notation and terminology. $\Omega^{-1}$ denotes the set $\{z \in \comp; \frac{1}{z} \in \Omega \} \subset \com+$, where $\Omega$ is the domain introduced in Section \ref{Lim}.  An entire function $f$ in $\comp$ is said to be of \textit{finite order} 
if there exist $0 \leq \rho < \infty$ and $C_1,C_2 > 0$ such that $|f(z)| \leq C_1 e^{C_2 |z|^{\rho}}$ for all $z \in \comp$. 
The infimum of such $\rho$ that the above $C_1,C_2$ exist is called the \textit{order} of $f$. If $f$ has an order $\rho$, then $f$ is said to be of \textit{finite type} if $\overline{\lim}_{a \to \infty}a^{-\rho}\log(\sup_{|z|\leq a}|f(z)|) < \infty$. The value of the limit superior is called the \textit{type} of $f$. We refer the reader to \cite{Lev} for more information on entire functions. %We note that an entire function $f$ is of order one an type $\sigma >0$ if and only if for any $\sigma' > \sigma$

\begin{thm}
Let $\mu$ be a probability measure on $\real$. The following conditions are equivalent. 
\begin{itemize}
\item[(1)] $\mu$ belongs to $\mathcal{P}_1$. 
\item[(2)] $G_\mu(\frac{1}{z})$ defined in $\com+$ has analytic continuation to $\{z \in \comp; |z| < r \}$ for some $r > 0$. 
\item[(3)] $F_\mu(\frac{1}{z}) - \frac{1}{z}$ defined in $\com+$ has analytic continuation to $\{z \in \comp; |z| < r \}$ for some $r > 0$. 
\item[(4)] $F_\mu^{-1}(\frac{1}{z})-\frac{1}{z}$ defined in $\Omega^{-1}$ has analytic continuation to $\{z \in \comp; |z| < r \}$ for some $r > 0$. 
\item[(5)] The Fourier transform $\mathcal{F}_\mu(z)$ defined in $z > 0$ is the restriction of an entire function $\widetilde{\mathcal{F}}_\mu(z)$ of order one and of finite type. 
\end{itemize}
Moreover,  we define $r_\mu(k)$ to be the supremum of the radius $r > 0$ in the statement ($k$) for $k = 2,~3$ and  
$\widetilde{R}_\mu:=\inf \{R > 0; \sup_{z \in \comp} \frac{|\widetilde{\mathcal{F}}_\mu(z)|}{e^{R|z|}} < \infty \}$. Then  $r_\mu(k) = R_\mu^{-1}$ 
and $\widetilde{R}_\mu = R_\mu$. The second relation implies that the type of  $\widetilde{\mathcal{F}}_\mu(z)$ is $R_\mu$.   
\end{thm}
%\begin{rem}
%The characterization of (5) is a kind of Paley-Wiener theorem. We can prove a more sharp estimate than shown in $(5)$; the constant $R$ in (5) is closely related to the constant $R$ defined in Proposition \ref{prop13}. If we try to do so, however,  complexity arises due to the non-uniqueness of the decomposition of a measure in $\mathcal{P}_1$.  Therefore we do not work on it anymore in this paper.  
%\end{rem}
\begin{proof}
\textbf{1.}  We show the equivalence between (1) and (2). 
With Proposition \ref{prop1}, (1) implies (2). Conversely, we assume (2). Then we have the expansion 
\[
G_\mu \Big(\frac{1}{z}\Big) = \sum_{n=0}^\infty d_n z^{n+1} 
\]
for $|z| < r$. $d_0 = 1$ since $z G_\mu(z) \to 1$ in the non tangential limit $z \to \infty$, $z \in \comm$. 
%Let $a_n$ ($n \geq 1$) be a sequence of real numbers defined by $a_n=\frac{1}{\pi}\im d_{n-1}$. 
%Then we define $p(x)= \sum_{n=2}^\infty \frac{a_n}{x^n}$ for $|x| \geq R$ for any $R > r^{-1}$. 
%We denote by $m_n(p)$ the $n$-th complex moment of $p(x)dx$ on $|x| \geq R$; $m_n(p)$ is defined by 
%\[
%m_n(p) = \int_{C_R}\frac{p(\frac{1}{w})}{w^{n+2}}dw. 
%\]
%Now we prove that $p(x) \geq 0$. We can prove that $\im (m_n(p)) = \pi a_{n+1} = \im b_n$. 
%Therefore, we have $\lim_{y \searrow 0} \frac{1}{\pi}\im G_\mu (x - iy) = \sum_{n=0}^\infty \frac{a_{n+1}}{x^{n+1}} = p(x)$ locally uniformly on $|x| \geq R$. This implies that $\mu$ has an absolutely continuous density $p(x)$ on $|x| \geq R$. In particular, $p \geq 0$. 
%Therefore, $\mu$ is of the form $\mu = \mu|_{[-R, R]} + \lambda$, where $\lambda = p(x)dx$ on $|x| \geq R$.   

We take any $R > r^{-1}$. 
Clearly $\frac{1}{\pi}\im G_\mu (x - iy)$ converges to $p(x):=\sum_{n=1}^\infty \frac{\im d_n}{x^{n+1}}$ 
locally uniformly on $|x| \geq R$ as $y \searrow 0$. Then $\mu$ has an absolutely continuous density $p(x)$ on $|x| \geq R$ by using the Stieltjes-Perron inversion formula \cite{Akh}. Therefore,   
 $\mu$ is of the form $\mu = \mu|_{[-R, R]} + \lambda$, where $\lambda = p(x)dx$ on $|x| \geq R$. 

\textbf{2.}  The equivalence between (2) and (3) is immediate from simple computation. 

\textbf{3.}  The equivalence between (2) and (4) is proved as follows. Let $g(z):=G_\mu(\frac{1}{z})$ for $z \in \com+$. 
We have the identity 
\[
\frac{1}{g^{-1} (z)} - \frac{1}{z} = F^{-1}_\mu\Big(\frac{1}{z}\Big) - \frac{1}{z} 
\] 
for $z \in \Omega^{-1}$; this enables us to prove the equivalence. More precisely, we can show that (2) is equivalent to (2') which states that $g^{-1}(z)$ defined in $\Omega^{-1}$ has analytic continuation to $\{z \in \comp; |z| < r \}$ for some $r > 0$. Using this we can prove the equivalence between (2) and (4).  

\textbf{4.}  We prove the equivalence between (1) and (5). If (1) holds, the first statement in (5) follows from Proposition \ref{prop2}. The second follows from the estimation in Proposition \ref{prop13}. 

Conversely we assume that (5) holds. Let $c_n:= \frac{d^n \widetilde{F}}{d \xi^n}(0)$ for $n \geq 0$; in particular $c_0 = 1$. Then for each $A >R$, there exists $C$ depending on $A$ such that $|c_n| \leq CA^n$ for all $n$. This is proved as follows. 
The Cauchy's integral formula implies that $|c_n| \leq \frac{n! e^x R^n}{x^n}$ for any $x > 0$ and $n$. Since $\frac{e^x}{x^n}$ takes its minimum value at $n$, we take $x=n$. If $A > R$, we can prove that $|c_n| < A^n$ for large $n$ applying Stirling's approximation.   
 
For $t > R$ and $N > 0$ we define $f^{+}_N (t)$ by 
\[
f_N^{+}(t) = \frac{1}{2\pi i} \int_0 ^N e^{-t\xi} \widetilde{\mathcal{F}}_\mu (-i\xi)d\xi. 
\] 
$f_N^{+}(t)$ converges to $f^{+}(t):= \frac{1}{2\pi i}\sum_{n=0}^\infty \frac{c_n}{t^{n+1}}$ as $N \to \infty$ locally uniformly in $(R, \infty)$. 

Changing the path,  we obtain 
\[
f_N^{+}(t)  = \frac{1}{2\pi}\int_0 ^N e^{-it\xi} \mathcal{F}_\mu (\xi)d\xi -  \int_{\Gamma_N} e^{-t\xi} \widetilde{\mathcal{F}}_\mu (-i\xi)d\xi,   
\]
where $\Gamma_N$ is a curve defined by $\Gamma_N = \{Ne^{i\theta}; 0 \leq \theta \leq \frac{\pi}{2}\}$. The integral over the path 
$\Gamma_N$ converges to $0$ as $N \to \infty$ locally uniformly.  Therefore, we obtain 
\[
\frac{1}{2\pi}\int_0 ^N e^{-it\xi} \mathcal{F}_\mu (\xi)d\xi \to  \frac{1}{2\pi i}\sum_{n=0}^\infty \frac{c_n}{t^{n+1}} 
\]
locally uniformly in $(R, \infty)$. By summing up the complex conjugate,  
\[
\frac{1}{2\pi}\int_{-N} ^N e^{-it\xi} \mathcal{F}_\mu (\xi)d\xi  \to  \frac{1}{\pi}\sum_{n=0}^\infty \frac{\im c_n}{t^{n+1}} 
\]
locally uniformly. Similarly we can prove the above convergence for $t < -R$: we have only to replace $f_N^{+}(t)$ by 
$f_N^{-}(t):= \frac{i}{2\pi } \int_0 ^N e^{t\xi} \widetilde{\mathcal{F}}_\mu (i\xi)d\xi$.  
Now L\'{e}vy's inversion formula (see Theorem 6.2.1 in \cite{KLC} for instance) implies that $\mu$ has the absolutely continuous density $\frac{1}{\pi}\sum_{n=0}^\infty \frac{\im c_n}{t^{n+1}}$ 
for $|t| > R$, which completes the proof.  

\textbf{5.}  The equalities of $r_\mu(k)$ and $\widetilde{R}_\mu$ are easily proved by using the above proofs of \textbf{1}, \textbf{2}, \textbf{4} 
and Theorem \ref{thm1} and Proposition \ref{prop13}.  
\end{proof}
\begin{cor}\label{cor1}
$\mathcal{P}_1$ is closed under the convolutions $\ast$, $\boxplus$, $\uplus$ and $\rhd$. 
\end{cor}
\begin{proof}
This is immediate since the convolutions are characterized as (\ref{ast})-(\ref{rhd}). 
\end{proof}

%complex moments are applicable to prove the corresponding theorem in free probability. In this case we use the class $\mathcal{P}_1$. 
%\begin{defi}
%We say that a sequence $\{a_n \}_{n=1}^\infty \subset \comp$ does not grow faster than exponentially if there exists $c > 0$ such that 
%$|a_n| \leq c^n$. 
%\end{defi}
%The following properties are known.  
%\begin{prop} Let $a_n$ be a sequence of complex numbers and $K_n^F$, $K_n^B$ be the respective free and Boolean cumulants.  
%(1) $a_n$ does not grow faster than exponentially if and only if $K_n^F$ does not grow faster than exponentially. \\
%(2) $a_n$ does not grow faster than exponentially if and only if $K_n^B$ does not grow faster than exponentially. 
%\end{prop}
%\begin{rem}
%(1) This is not the case for tensor cumulants $K_n^T$. In fact, it is known that infinitely divisible distributions except for delta measures have 
%unbounded supports. Therefore, if $K_n^T(\mu)$ comes from an infinitely divisible distribution $\mu$ whose L\'{e}vy measure has a bounded support, 
%$K_n^T$ does not grow faster than exponentially, but there is no $c > 0$ such that $|m_n(\mu)| \leq c^n$ for all $n$. \\
%(2) ??????????It is not known if this is the case for monotone cumulants. If this is the case, we can obtain a new limit theorem concerning Cauchy distributions. 
%\end{rem}

\section*{Acknowledgements} 
The author would like to express his sincere thanks to Mr. Hayato Saigo for many suggestions of importance of stable distributions,  
especially Cauchy distributions, and for discussions of limit theorems, and also suggestions of further directions in the future. 
The author is also grateful very much to Professor Izumi Ojima for informing him many uses of Cauchy distributions in Physics, and for suggestions of applications 
of complex moments to Physics, which motivated the author greatly. Also he would like to thank Mr. Kazuya Okamura for discussions 
about Cauchy distributions in statistics and applications of complex moments to statistics. 
This work was supported by JSPS (KAKENHI 21-5106).

\end{document}